\documentclass[12pt, reqno]{amsart}
\usepackage{amsmath, amsthm, amscd, amsfonts, amssymb, graphicx, color}
\usepackage[bookmarksnumbered, colorlinks, plainpages]{hyperref}

\newtheorem{thm}{Theorem}[section]
\newtheorem{lem}{Lemma}[section]

\newtheorem{cor}{Corollary}[section]
\newtheorem{prop}{Proposition}[section]
\newtheorem{defn}{Definition}[section]

\numberwithin{equation}{section}

\begin{document}

\begin{center}
{\large{\textbf{CHARACTERISTICS OF CONFORMAL RICCI SOLITON ON WARPED PRODUCT SPACES}}}
\end{center}
\vspace{0.1 cm}
\begin{center}By\end{center}
\begin{center}
{Dipen~~Ganguly\footnote{The first author D. Ganguly is the corresponding author and is thankful to the National Board for Higher Mathematics (NBHM), India(Ref No: 0203/11/2017/RD-II/10440) for their financial support to carry on this research work.},Nirabhra~~Basu$^2$ and Arindam~~Bhattacharyya$^3$}
\end{center}
\vskip 0.3cm
\begin{center}
$^{1,3}$Department~of~Mathematics\\
Jadavpur~University,\\
Kolkata-700032,~India.\\
E-mail: dipenganguly1@gmail.com\\
E-mail: bhattachar1968@yahoo.co.in\\
$^2$Department~of~Mathematics\\
Bhawanipur~Education~Society~College,\\
Kolkata-700020,~India.\\
E-mail: nirabhra.basu@thebges.edu.in\\
\end{center}
\vskip 0.3cm

\begin{center}
\textbf{Abstract}\end{center}\par
\medskip
Conformal Ricci solitons are self similar solutions of the conformal Ricci flow equation. This paper deals with the study of conformal Ricci solitons within the framework of warped product manifolds which extends the notion of usual Riemannian product manifolds. First, we prove that if a warped product manifold admits conformal Ricci soliton then the base and the fiber also share the same property. In the next section the characterization of conformal Ricci solitons on warped product manifolds in terms of Killing and conformal vector fields has been studied. Next, we prove that a warped product manifold admitting conformal Ricci soliton with concurrent potential vector field is Ricci flat. Finally, an application of conformal Ricci soliton on a class of warped product spacetimes namely, generalized Robertson-Walker spacetimes has been discussed.\par
\medskip
\begin{flushleft}
\par
\end{flushleft}
\medskip
\begin{flushleft}
 \par
\end{flushleft}
\medskip
\medskip
\medskip
\section{\textbf{Introduction}}
Hamilton's \cite{[3]},\cite{[4]} theory of Ricci flow reached to a highest magnitude and popularity after G.Perelman \cite{[1]},\cite{[2]} successfully applied it in solving the Poincar\'{e} conjecture.  The study of Ricci solitons was also introduced by Hamilton as fixed or stationary points of the Ricci flow in the space of the parameterized metrics ${g(t)}$ on $M$ modulo diffeomorphisms and scaling. Since then, both the topics have been studied by many mathematicians like Brendle \cite{[13]}, H. D. Cao \cite{[14]}, B.Y. Chen \cite{[9]} etc. and many others \cite{[5]},\cite{[15]},\cite{[21]},\cite{[22]}.
\par
\medskip
A smooth manifold $M$ equipped with a Riemannian metric $g$ is said to be a Ricci soliton, if for some constant $\lambda$, there exist a smooth vector field $X$ on $M$ satisfying the equation
\begin{equation}
Ric+\frac{1}{2}\mathcal{L}_{X}g=\lambda g,\nonumber
\end{equation}
where $\mathcal{L}_{X}$ denotes the Lie derivative and $Ric$ is the Ricci tensor. The Ricci soliton is called shrinking if $\lambda >0$, steady if $\lambda =0$ and expanding if $\lambda <0$.\par
\medskip
 A. E. Fischer \cite{[6]} has introduced conformal Ricci flow in 2004 as a modified Ricci flow.  It preserves the constant scalar curvature of the evolving metrics. Because of the role of conformal geometry plays in maintaing scalar curvature constant such a modified Ricci flow was named as conformal Ricci flow.\par
\medskip
 The conformal Ricci flow equation on a smooth closed connected oriented n-manifold, $n\geq 3$, is given by
\begin{equation}
\frac{\partial g}{\partial t}+2(Ric+\frac{g}{n})=-pg~~~and~~~r(g)=-1, \nonumber
\end{equation}
where $p$ is a non-dynamical(time dependent) scalar field and $r(g)$ is the scalar curvature of the manifold.The term $-pg$ acts as the constraint force to maintain the scalar curvature constraint. Thus these evolution equations are analogous to famous Navier-Stokes equations where the constraint is divergence free. For this reason $p$ is also called the conformal pressure.\par
\medskip
In 2015, Basu and Bhattacharyya \cite{[7]} introduced the notion of conformal Ricci soliton equation as
\begin{equation}
\mathcal{L}_{X}g+2Ric=[2\lambda -(p+\frac{2}{n})]g,
\end{equation}
where $\lambda$ is constant and $p$ is the conformal pressure. The equation is the generalization of the Ricci soliton equation and it satisfies the conformal Ricci flow equation.\par
\medskip
After N. Basu's paper \cite{[7]}, many authors have studied conformal Ricci solitons and obtained interesting results. D. Ganguly and A. Bhattacharyya studied the conformal Ricci soliton within the framework of almost co-K\"{a}hler manifolds \cite{DG1} and $(LCS)_n$-manifolds \cite{DG2} and obtained some interesting results. T. Dutta et. al. in \cite{[20]} showed that if a Lorentzian $\alpha$-Sasakian manifold admits conformal Ricci soliton and is Weyl conformally semi-symmetric, then the manifold is $\eta$-Einstein. Again in \cite{[18]} it has been proved that a $(k,\mu)^{'}$-almost Kenmotsu manifold admitting conformal Ricci soliton is either locally isometric to $\mathbb{H}^{n+1}(-4)\times\mathbb{R}^n$ or the soliton is expanding, steady or shrinking according to some conditions on the conformal pressure $p$. S.K. Hui et. al. \cite{[19]}  proved that a conformal Ricci soliton on a three-dimensional $f$-Kenmotsu manifold with torqued potential vector field is an almost quasi-Einstein manifold. In \cite{MDS} M. D. Siddiqui et al., established that a conformal Ricci soliton, on a perfect fluid spacetime with torse-forming vector field and without cosmological constant, is expanding. So, motivated by the above studies, in this paper we investigate the nature of conformal Ricci soliton on warped product spaces.\par
\medskip
The concept of warped product was introduced by Bishop and O'Neill \cite{[8]}. They have given the examples of complete Riemannian manifolds with negative sectional curvature. After that the study of warped product have been of great interest among both mathematicians and physicists. Considering two Riemannian manifolds $(B,g_B)$ and $(F,g_F)$ as well as a positive smooth function $f$ on $B$,we define on the product manifold $B\times F$, the metric
\begin{equation}
g= \pi^*g_B+(f\circ\pi)^2\sigma^*g_F,
\end{equation}
where $\pi$ and $\sigma$ are the natural projections on $B$ and $F$ respectively. Under this conditions the product manifold is said to be the warped product of $B$ and $F$; it is denoted by $M=B\times_f F$. Here the manifold $B$ is called the base manifold and $F$ is called the fiber. The function $f$ is called the warping function.\par
\medskip
In our study, we shall focus on the conditions that makes the warped product to a conformal Ricci soliton.
To begin with, let us first recall a very important result (for details see \cite{[8]}) which will be required for our purpose in later sections.\par
\medskip
\begin{lem}
  Let $(M,g)=(B\times_fF,g_B\oplus f^2g_F)$ be an warped product of two Riemannian manifolds $(B,g_B)$ and $(F,g_F)$ with dim$B=m$ and dim$F=n$. Then for all $X,Y\in\mathfrak{X}(B)$ and $U,V\in\mathfrak{X}(F)$
\begin{enumerate}
    \item $D_XU=D_UX=\frac{X(f)}{f}U$,
    \item $Ric(X,U)=0$,
    \item $Ric(X,Y)=Ric^B(X,Y)-\frac{n}{f}H^f(X,Y)$,
    \item $Ric(U,V)=Ric^F(U,V)-(\frac{\Delta f}{f}+(n-1)\frac{\|\nabla f\|^2}{f^2})g(U,V)$,
\end{enumerate}
  where $D_XY$ is the lift of $\nabla_XY$ on $B$ and $Ric^B$, $Ric^F$ are the lifts of the Ricci tensors on the base $B$ and the fiber $F$ respectively.
\end{lem}
\medskip
\medskip
\section{\textbf{Conformal Ricci soliton on warped product manifolds}}
This section deals with the investigations of conformal Ricci soliton on warped product manifolds. Basically here we want to study if a warped product manifold admits a conformal Ricci soliton then how its effect is on the base manifold and on the fiber i.e; we try to find out under which conditions they become conformal Ricci soliton. So, for this purpose let us first suppose that $(M,g)=(B\times_fF,g_B\oplus f^2g_F)$ be an warped product of two Riemannian manifolds $(B,g_B)$ and $(F,g_F)$ with dim$B=m$ and dim$F=n$. Now let $(M,g,\mu,\xi)$ be a conformal Ricci soliton, where $\mu=[2\lambda -(p+\frac{2}{n})]$. Then from equation $(1.1)$ we get
\begin{equation}
\mathcal{L}_{\xi}g+2Ric=[2\lambda -(p+\frac{2}{n})]g=\mu g
\end{equation}
where $\mu=[2\lambda -(p+\frac{2}{n})]$.
Again note that from the lemma $1.1$ of the previous section (for details see \cite{[8]}) the following two well-known formulas for warped product manifolds can be easily deduced
\begin{eqnarray}
  \mathcal{L}_{\xi}g &=& {\mathcal{L}}^B_{\xi_B}g_B+f^2{\mathcal{L}}^F_{\xi_F}g_F+2f\xi_B(f)g_F, \\
  Ric &=& Ric^B-\frac{n}{f}H^f+Ric^F-\tilde{f}g_F,
\end{eqnarray}
where $\tilde{f}=f\Delta f+(n-1)\|\nabla f\|^2_B$.
Now in equation $(2.1)$ using the definition of the warped metric from equation $(1.2)$ and then applying the values from the above two equations $(2.2)$ and $(2.3)$ we have
\begin{eqnarray}
  \mu (g_B+f^2g_F) &=& \mu g \nonumber\\
   &=& \mathcal{L}_{\xi}g+2Ric \nonumber\\
   &=& {\mathcal{L}}^B_{\xi_B}g_B+f^2{\mathcal{L}}^F_{\xi_F}g_F+2f\xi_B(f)g_F+2Ric^B \nonumber\\
   && -2\frac{n}{f}H^f+2Ric^F-2\tilde{f}g_F,
\end{eqnarray}
\par
\medskip
Again for all $U,V\in\mathfrak{X}(B)$, using the definition of Lie derivation we can write
\begin{equation}
  ({\mathcal{L}}^B_{\xi_B}g_B)(U,V)=g_B(D^B_U\xi_B,V)+g_B(U,D^B_V\xi_B).
\end{equation}
Now from the definition of Hessian and the above equation $(2.5)$ we have
\begin{equation}
  ({\mathcal{L}}^B_{\xi_B}g_B-2\frac{n}{f}H^f)(U,V)=g_B(D^B_U\xi_B,V)+g_B(U,D^B_V\xi_B)-2\frac{n}{f}g_B(D^B_U\nabla^Bf,V).\nonumber
\end{equation}
The above equation can be rewritten as
\begin{eqnarray}
  ({\mathcal{L}}^B_{\xi_B}g_B-2\frac{n}{f}H^f)(U,V) &=& (g_B(D^B_U\xi_B,V)-\frac{n}{f}g_B(D^B_U\nabla^Bf,V)) \nonumber\\
                                                    && +(g_B(U,D^B_V\xi_B)-\frac{n}{f}g_B(D^B_U\nabla^Bf,V)) \nonumber\\
                                                    &=& g_B(D^B_U(\xi_B-n\nabla ^B\ln f),V) \nonumber\\
                                                    && +g_B(U,D^B_V(\xi_B-n\nabla ^B\ln f)).
\end{eqnarray}
Using the definition of Lie derivative again equation $(2.6)$ becomes
\begin{equation}
  ({\mathcal{L}}^B_{\xi_B}g_B-2\frac{n}{f}H^f)(U,V)=({\mathcal{L}}^B_{\xi_B-n\nabla ^B\ln f}g_B)(U,V), \forall U,V\in\mathfrak{X}(B).\nonumber
\end{equation}
Since the above equation is true for all $U,V\in\mathfrak{X}(B)$, in operator notation we can write
\begin{equation}
{\mathcal{L}}^B_{\xi_B}g_B-2\frac{n}{f}H^f={\mathcal{L}}^B_{\xi_B-n\nabla ^B\ln f}g_B
\end{equation}
Now using the value from equation $(2.7)$, the equation $(2.4)$ finally yields
\begin{multline}
  ({\mathcal{L}}^B_{\xi_B-n\nabla ^B\ln f}g_B+2Ric^B)+(f^2{\mathcal{L}}^F_{\xi_F}g_F+2Ric^F)\\=\mu g_B
    +(\mu f^2-2f\xi_B(f)+2\tilde{f})g_F).
\end{multline}
Hence from the above discussion and equation $(2.8)$ we have the following theorem\par
\medskip
\begin{thm}
  Let us consider that $(M,g)=(B\times_fF,g_B\oplus f^2g_F)$ be an warped product of two Riemannian manifolds $(B,g_B)$ and $(F,g_F)$ with warping function $f$, dim$B=m$ and dim$F=n$. If $(M,g,\mu,\xi)$ be a conformal Ricci soliton, then the base $(B,g_B,\mu,\xi_B-n\nabla ^B\ln f)$ and the fiber $(F,g_F,\mu f^2-2f\xi_B(f)+2\tilde{f},f^2\xi_F)$ are both conformal Ricci solitons; where $\tilde{f}=f\Delta f+(n-1)\|\nabla f\|^2_B$, $\mu=[2\lambda -(p+\frac{2}{n})]$, $\lambda$ is the soliton constant and $p$ is the conformal pressure.
\end{thm}
\par
\medskip
\medskip
Now we study a special case when the soliton vector field $\xi$ of the conformal Ricci soliton $(M,g,\mu,\xi)$ becomes gradient of some smooth function $\phi$ i.e; when $\xi=grad\phi=\nabla\phi$. In this case we call the soliton a conformal gradient Ricci soliton and the function $\phi$ is then called the potential function of the soliton. Also for notational purpose without any confusion we denote a conformal gradient Ricci soliton as $(M,g,\mu,\phi)$, where the last term specifies the potential function of the soliton. \par
\medskip
Let us assume that $(M,g)=(B\times_fF,g_B\oplus f^2g_F)$ be an warped product of two Riemannian manifolds $(B,g_B)$ and $(F,g_F)$ with dim$B=m$ and dim$F=n$. Then if $(M,g,\mu,\phi)$ be a conformal gradient Ricci soliton, for any vector fields $X,Y\in\mathfrak{X}(M)$, equation $(1.1)$ implies
\begin{equation}
  2H^{\phi}(X,Y)+2Ric(X,Y)=[2\lambda -(p+\frac{2}{n})]g(X,Y)=\mu g(X,Y).
\end{equation}
\par
\medskip
 Now if we take $X=X_B$ and $Y=Y_B$, where $X_B,Y_B$ are the lifts of the vector fields $X,Y$ in $\mathfrak{X}(B)$, then the equation $(2.9)$ gives us
\begin{equation}
   2H^{\phi}(X_B,Y_B)+2Ric(X_B,Y_B)=\mu g(X_B,Y_B).\nonumber
\end{equation}
Using the value of the Ricci tensor for the base manifold from lemma $1.1$, the above equation becomes
\begin{equation}
   2H^{\phi _B}_B(X_B,Y_B)+2Ric^B(X_B,Y_B)-2\frac{n}{f}H^f_B(X_B,Y_B)=\mu g_B(X_B,Y_B),\nonumber
\end{equation}
where $\phi_B=\phi$ at a fixed point of the fiber $F$. Finally using the properties of Hessian in the above equation we get
\begin{equation}
   2H^{\phi _B-n\ln f}_B(X_B,Y_B)+2Ric^B(X_B,Y_B)=\mu g_B(X_B,Y_B).
\end{equation}
This shows that $(B,g_B,\mu,\phi _B-n\ln f)$ is a conformal gradient Ricci soliton.\par
\medskip
Again taking $X=X_F$ and $Y=Y_F$, where $X_F,Y_F$ are the lifts of the vector fields $X,Y$ in $\mathfrak{X}(F)$, then the equation $(2.9)$ gives us
\begin{equation}
   2H^{\phi}(X_F,Y_F)+2Ric(X_F,Y_F)=\mu g(X_F,Y_F).\nonumber
\end{equation}
Using equation $(2.3)$ and lemma $1.1$ the above equation becomes
\begin{equation}
   2H^{\phi _F}_F(X_F,Y_F)+2Ric^F(X_F,Y_F)-\tilde{f}g_F(X_F,Y_F)=\mu f^2g_F(X_F,Y_F),\nonumber
\end{equation}
where $\phi_F=\phi$ at a fixed point of the base $B$ and $\tilde{f}=f\Delta f+(n-1)\|\nabla f\|^2_B$. Thus finally we get from the above equation
\begin{equation}
   2H^{\phi _F}_F(X_F,Y_F)+2Ric^F(X_F,Y_F)=(\mu f^2+\tilde{f})g_F(X_F,Y_F).\nonumber
\end{equation}
Therefore if the warping function $f$ is constant, the term $\tilde{f}=f\Delta f+(n-1)\|\nabla f\|^2_B$ vanishes from the right hand side of the above equation and we get the following\par
\medskip
\begin{equation}
   2H^{\phi _F}_F(X_F,Y_F)+2Ric^F(X_F,Y_F)=\mu f^2g_F(X_F,Y_F).
\end{equation}
Thus $(F,g_F,\mu f^2,\phi_F)$ is a conformal gradient Ricci soliton. Hence from the above observations and equations $(2.10)$ and $(2.11)$ we can state the following
\begin{thm}
  Let $(M,g)=(B\times_fF,g_B\oplus f^2g_F)$ be an warped product of two Riemannian manifolds $(B,g_B)$ and $(F,g_F)$ with warping function $f$, dim$B=m$ and dim$F=n$. If $(M,g,\mu,\phi)$ be a conformal gradient Ricci soliton, then
   \begin{enumerate}
     \item the base $(B,g_B,\mu,\phi _B-n\ln f)$ is a conformal gradient Ricci soliton with $\phi_B=\phi$ at a fixed point of the fiber $F$.
     \item the fiber $(F,g_F,\mu f^2,\phi_F)$ is a conformal gradient Ricci soliton with $\phi_F=\phi$ at a fixed point of the base $B$, provided the warping function $f$ is constant.
   \end{enumerate}
   \end{thm}
\par
\medskip
\medskip
\section{\textbf{Effect of certain special types of vector fields on conformal Ricci soliton on warped product manifolds}}
The main purpose of this section is to study the effects of some special types of smooth vector fields on conformal Ricci solitons on warped product spaces. In particular, we will focus on Killing vector fileds, conformal vector fields and concurrent vector fields. So, We have included some necessary definitions before proceeding further.\par
\medskip
\begin{defn}
  A smooth vector field $X$ on a Riemannian manifold $(M,g)$ is called
  \begin{enumerate}
    \item a Killing vector field or an infinitesimal isometry, if the local $1$-parameter group of transformations generated by $X$ in a neighbourhood of each point of $M$ consists of local isometries, or in otherwords, if $X$ satisfies $\mathcal{L}_{X}g=0$ and
    \item a conformal vector field if $X$ satisfies $\mathcal{L}_{X}g=\rho g$,
  \end{enumerate}
  where $\rho$ is a smooth function on the manifold $M$ and $\mathcal{L}_{X}g$ denotes the Lie differentiation of the Riemannian metric $g$ in the direction of the vector field $X$.
\end{defn}
\par
\medskip
So, based on the above definition our first result of this section is the following
\begin{prop}
  Let $(M,g)=(B\times_fF,g_B\oplus f^2g_F)$ be an warped product of two Riemannian manifolds $(B,g_B)$ and $(F,g_F)$ with warping function $f$, dim$B=m$ and dim$F=n$. If $(M,g,\mu,\xi)$ is a conformal Ricci soliton and any one of the following conditions hold
  \begin{enumerate}
    \item $\xi=\xi_B$ and $\xi_B$ is a Killing vector field on the base $B$.
    \item $\xi=\xi_F$ and $\xi_F$ is a Killing vector field on the fiber $F$.
  \end{enumerate}
  Then the manifold $(M,g)$ becomes an Einstein manifold.
\end{prop}
\begin{proof}
As per our assumption $(M,g,\mu,\xi)$ being a conformal Ricci soliton, it satisfies equation $(1.1)$ and we get
\begin{equation}
\mathcal{L}_{\xi}g+2Ric=\mu g,
\end{equation}
 Now let, $\xi=\xi_B$, and $\xi_B$ is Killing on $B$, we get ${\mathcal{L}}^B_{\xi_B}g_B=0$. Then using it in equation $(2.2)$ we have $\mathcal{L}_{\xi}g=0$. Therefore equation $(3.1)$ gives us $Ric=\frac{\mu}{2}g$ and this implies $(M,g)$ is Einstein manifold.\par
  \medskip
  Again if $\xi=\xi_F$ and $\xi_F$ is a Killing vector field on the fiber $F$, ${\mathcal{L}}^F_{\xi_F}g_F=0$. Then using equations $(2.2)$ and $(3.1)$ and proceeding similarly as the first part of the proof, it can be easily shown that in this case also $(M,g)$ is Einstein.

\end{proof}
\begin{thm}
  Let $(M,g)=(B\times_fF,g_B\oplus f^2g_F)$ be an warped product of two Riemannian manifolds $(B,g_B)$ and $(F,g_F)$ with warping function $f$, dim$B=m$ and dim$F=n$. If $(M,g,\mu,\xi)$ is a conformal Ricci soliton and $\xi_B$ is Killing vector field on the base $B$;  then the base $(B, g_B, \mu, -n\ln f)$ is a conformal gradient Ricci soliton; where $\xi_B$ is the lift of the vector field $\xi$ to $\mathfrak{X}(B)$.
\end{thm}
\begin{proof}
  Since it is given that $(M,g,\mu,\xi)$ is a conformal Ricci soliton from theorem $2.1$ it follows that the base $(B,g_B,\mu,\xi_B-n\nabla ^B\ln f)$ is also a conformal Ricci soliton and hence it satisfies equation $1.1$. Thus we can write
  \begin{equation}
    {\mathcal{L}}^B_{\xi_B-n\nabla ^B\ln f}g_B+2Ric^B=\mu g_B.
  \end{equation}
Again using equation $(2.7)$ the above equation $(3.2)$ becomes
\begin{equation}
    {\mathcal{L}}^B_{\xi_B}g_B-2\frac{n}{f}H^f+2Ric^B=\mu g_B.\nonumber
  \end{equation}
Now, as $\xi_B$ is Killing vector field on the base $B$, we have ${\mathcal{L}}^B_{\xi_B}g_B=0$. Thus with the help of this, the above equation gives us
\begin{equation}
    -2\frac{n}{f}H^f+2Ric^B=\mu g_B.\nonumber
  \end{equation}
  Thus using the properties of Hessian, the above equation finally yields
  \begin{equation}
    2H^{-n\ln f}+2Ric^B=\mu g_B.
  \end{equation}
  Hence compairing the above equation $(3.3)$ with the conformal gradient Ricci soliton equation $(2.9)$ completes the proof.
\end{proof}
\par
\medskip
We conclude this portion of study of Killing vector fields on conformal Ricci soliton warped product manifolds with the following result
\begin{thm}
   Assume that $(M,g)=(B\times_fF,g_B\oplus f^2g_F)$ be an warped product of two Riemannian manifolds $(B,g_B)$ and $(F,g_F)$ with warping function $f$, dim$B=m$ and dim$F=n$. Let $(M,g,\mu,\xi)$ be a conformal Ricci soliton and both the lifts $\xi_B$ and $\xi_F$ are Killing on the base $B$ and the fiber $F$ respectively. Then the manifold $(M,g)$ is Einstein if $\xi_B(f)=0$.
\end{thm}
\begin{proof}
  Since it is given that both $\xi_B$ and $\xi_F$ are Killing, we have ${\mathcal{L}}^B_{\xi_B}g_B=0$ and ${\mathcal{L}}^F_{\xi_F}g_F=0$. Then using these values in equation $(2.2)$ we get
  \begin{equation}
  \mathcal{L}_{\xi}g=2f\xi_B(f)g_F.
  \end{equation}
  Again as per our hypothesis $(M,g,\mu,\xi)$ being a conformal Ricci soliton, from equation $(1.1)$ we get
  \begin{equation}
\mathcal{L}_{\xi}g+2Ric=\mu g.\nonumber
\end{equation}
Now using equation $(3.4)$ in the above equation, gives us
 \begin{equation}
2f\xi_B(f)g_F+2Ric=\mu g.
\end{equation}
Thus if $\xi_B(f)=0$, the above equation $(3.5)$ yields $Ric=\frac{\mu}{2}g$, which implies the manifold $(M,g)$ is Einstein and this completes the proof.
\end{proof}
Now we shall focus on the effect of conformal vector fields on warped product manifolds admitting conformal Ricci solitons. In this direction a very immediate result is the following
\begin{prop}
  Let $(M,g)=(B\times_fF,g_B\oplus f^2g_F)$ be an warped product of two Riemannian manifolds $(B,g_B)$ and $(F,g_F)$ with warping function $f$, dim$B=m$ and dim$F=n$. Let $(M,g,\mu,\xi)$ is a conformal Ricci soliton. Then the manifold $(M,g)$ becomes an Einstein manifold with factor $(\frac{\mu}{2}-\rho)$ if and only if the vector field $\xi$ is conformal with factor $2\rho$.
\end{prop}
\begin{proof}
  $(M,g,\mu,\xi)$ being a conformal Ricci soliton, from $(1.1)$ we can write
  \begin{equation}
\mathcal{L}_{\xi}g+2Ric=\mu g
\end{equation}
  Since the vector field $\xi$ is conformal with factor $2\rho$, by definition $3.1$ we have $\mathcal{L}_{X}g=2\rho g$, where $\rho$ is a smooth function. Thus using this value in equation $(3.6)$ finally we get
   \begin{equation}
Ric=(\frac{\mu}{2}-\rho)g.
\end{equation}
This implies $(M,g)$ is an Einstein manifold. Similarly by reverse calculation process it can be shown that if $(M,g)$ is an Einstein manifold with factor $(\frac{\mu}{2}-\rho)$ then $\xi$ becomes conformal with factor $2\rho$. This completes the proof.
\end{proof}
It is to be noted that in the above result we have discussed on conformal Ricci solitons with the vector field $\xi$ is taken conformal. So it is natural to ask whether it is necessary to consider $\xi$ conformal as a whole, or is there a weaker condition than this. The following theorem could put some light on it.
\begin{thm}
   Assume that $(M,g)=(B\times_fF,g_B\oplus f^2g_F)$ be an warped product of two Riemannian manifolds $(B,g_B)$ and $(F,g_F)$ with warping function $f$, dim$B=m$ and dim$F=n$. Let $(M,g,\mu,\xi)$ be a conformal Ricci soliton and both the lifts $\xi_B$ and $\xi_F$ are conformal on the base $B$ and the fiber $F$ with factors $2\rho_B$ and $2\rho_F$ respectively; where $\rho_B$ and $\rho_F$ are two smooth functions. Then the manifold $(M,g)$ is Einstein provided $\rho_B=\rho_F+\xi_B(\ln f)$.
\end{thm}
\begin{proof}
Since $\xi_B$ is conformal on the base $B$ with factor $2\rho_B$, we have ${\mathcal{L}}^B_{\xi_B}g_B=2\rho_Bg_B$. Also $\xi_F$ being conformal with factor $2\rho_F$, we get ${\mathcal{L}}^F_{\xi_F}g_F=2\rho_Fg_F$. Then  using these two values in equation $(2.2)$ we get
\begin{equation}
\mathcal{L}_{\xi}g=2(\rho_Bg_B+f^2\rho_Fg_F+f\xi_B(f)g_F).
\end{equation}
Again $(M,g,\mu,\xi)$ being a conformal Ricci soliton, from equation $(1.1)$ and the above equation $(3.8)$ we have
\begin{equation}
2(\rho_Bg_B+f^2\rho_Fg_F+f\xi_B(f)g_F+Ric)=\mu g.\nonumber
\end{equation}
The above equation can be rewritten as
\begin{equation}
Ric=\frac{\mu}{2}g-\rho_Bg_B-f^2(\rho_F+\xi_B(\ln f))g_F.
\end{equation}
Hence if $\rho_B=\rho_F+\xi_B(\ln f)$, and using equation $(1.2)$, the above equation $(3.9)$ finally gives us $Ric=(\frac{\mu}{2}-\rho_B)g$. This implies $(M,g)$ is Einstein and thus completes the proof.
\end{proof}
We end this section with our last theorem, which actually gives the converse part of the previous theorem. In the previous result we characterised the conformal Ricci soliton $(M,g,\mu,\xi)$ whereas our next result gives conditions under which a warped product manifold $(M,g)$ admits a conformal Ricci soliton.
\begin{thm}
  Let $(B,g_B,\mu,\xi_B)$ be a conformal Ricci soliton and $(F,g_F)$ be an Einstein manifold with factor $\beta$, where dim$B=m$ and dim$F=n$. Let $(M,g)=(B\times_fF,g_B\oplus f^2g_F)$ be an warped product of $(B,g_B)$ and $(F,g_F)$ with warping function $f$ and $\xi_F$ is conformal vector field with factor $2\rho$. Then $(M,g,\mu,\xi)$ is a conformal Ricci soliton if $H^f=0$ and the warping function $f$ satisfies the quadratic equation
  \begin{equation}
    (2\rho-\mu)f^2+2f\xi_B(f)+2\beta+2(1-n)k^2=0,\nonumber
  \end{equation}
  where $k^2=\|\nabla f\|^2_B=g_B(\nabla f,\nabla f)$ for some real number $k$.
\end{thm}
\begin{proof}
  $(B,g_B,\mu,\xi_B)$ being a conformal Ricci soliton, from equation $(1.1)$ we get
  \begin{equation}
\mathcal{L}^B_{\xi_B}g_B+2Ric^B=\mu g_B.
\end{equation}
Again as $(F,g_F)$ is an Einstein manifold with factor $\beta$, the Ricci tensor is given by $Ric^F=\beta g_F$. Using this value in equation $(2.3)$ gives us
\begin{equation}
  Ric=Ric^B-\frac{n}{f}H^f+\beta g_F-\tilde{f}g_F,
\end{equation}
where $\tilde{f}=f\Delta f+(n-1)\|\nabla f\|^2_B$. Now, using equation $(3.10)$ in the equation $(2.2)$ we get
\begin{equation}
\mathcal{L}_{\xi}g=\mu g_B-2Ric^B+f^2{\mathcal{L}}^F_{\xi_F}g_F+2f\xi_B(f)g_F.
\end{equation}
Multiplying both sides of the equation $(3.11)$ by $2$ and then adding it with equation $(3.12)$ yields
\begin{equation}
\mathcal{L}_{\xi}g+2Ric=\mu g_B+f^2{\mathcal{L}}^F_{\xi_F}g_F+2f\xi_B(f)g_F+2(-\frac{n}{f}H^f+\beta g_F-\tilde{f}g_F).\nonumber
\end{equation}
Now since the vector field $\xi_F$ is conformal with factor $2\rho$ i.e; ${\mathcal{L}}^F_{\xi_F}g_F=2\rho g_F$, the above equation becomes
\begin{equation}
\mathcal{L}_{\xi}g+2Ric=\mu g_B+2f^2\rho g_F+2f\xi_B(f)g_F+2(-\frac{n}{f}H^f+\beta g_F-\tilde{f}g_F).
\end{equation}
As it is given that $H^f=0$, then it implies that $\Delta f=0$ and hence $\tilde{f}=f\Delta f+(n-1)\|\nabla f\|^2_B$ becomes $\tilde{f}=(n-1)\|\nabla f\|^2_B=(n-1)k^2$,  where $k^2=\|\nabla f\|^2_B=g_B(\nabla f,\nabla f)$ for some real number $k$. Therefore using these results in the above equation $(3.13)$ we get
\begin{eqnarray}
\mathcal{L}_{\xi}g+2Ric &=& \mu g_B+2f^2\rho g_F+2f\xi_B(f)g_F+2(\beta g_F-(n-1)k^2g_F) \nonumber\\
&=& \mu(g_B+f^2g_F)+(2f^2\rho-\mu f^+2f\xi_B(f)+2(\beta-(n-1)k^2))g_F. \nonumber
\end{eqnarray}
Thus if $(2f^2\rho-\mu f^2+2f\xi_B(f)+2(\beta-(n-1)k^2))=0$ i.e; if $f$ satisfies the quadratic equation $(2\rho-\mu)f^2+2f\xi_B(f)+2\beta+2(1-n)k^2=0$; the above equation finally becomes
\begin{equation}
\mathcal{L}_{\xi}g+2Ric=\mu(g_B+f^2g_F)=\mu g
\end{equation}
Therefore from equation $(3.14)$ we can conclude that $(M,g,\mu,\xi)$ is a conformal Ricci soliton and this completes the proof.
\end{proof}
\medskip
\medskip
\section{\textbf{Warped product manifolds admitting conformal Ricci soliton with concurrent vector field}}
K. Yano \cite{[16]} introduced concircular vector fields to study concircular mappings, which are basically conformal mappings that preserve geodesic circles. In Mathematical Physics and General Relativity concircular vector fields have many applications. B.Y. Chen in \cite{[17]} proved that a Loretzian manifold is a generalised Robertson-Walker spacetime if and only if it admits a timellike concircular vector field. A vector field $\xi$ on a Riemannian manifold $M$ satisfying
\begin{equation}
  \nabla_X\xi=\alpha X,
\end{equation}
for all vector fields $X\in\mathfrak{X}(M)$, is called a concircular vector field \cite{[9]}, where $\alpha$ is a non-trivial function on $M$. In particular, if the function $\alpha$ is constant one then the vector field $\xi$ is called a concurrent vector field. Thus we have the following definition \cite{[9]}
\begin{defn}
  A vector field $\xi$ on a Riemannian manifold $M$ is called a concurrent vector field if, for all vector fields $X\in\mathfrak{X}(M)$, the vector field $\xi$ satisfies the following equation
  \begin{equation}
  \nabla_X\xi=X.
\end{equation}
\end{defn}
Based on the above definition, in this section we study conformal Ricci solitons with the soliton vector field $\xi$ being concircular (also, concurrent) vector field. So in this direction our first result is as follows
\begin{thm}
    Let $(M,g,\mu,\xi)$ be a conformal Ricci soliton on an $n$-dimensional Riemannian manifold $(M,g)$  and the soliton vector field $\xi$ is concircular with factor $\alpha$, then
     \begin{enumerate}
       \item the manifold $(M,g)$ is an Einstein manifold with factor $(\mu-2\alpha)$ and
       \item the soliton is expanding, steady or shrinking according as $(p+2\alpha+\frac{1}{n})<0$, $(p+2\alpha+\frac{1}{n})=0$ or $(p+2\alpha+\frac{1}{n})>0$ respectively.
     \end{enumerate}
\end{thm}
\begin{proof}
As per our assumption the soliton vector field $\xi$ is concircular with factor $\alpha$, then from equation $(4.1)$ we get $\nabla_X\xi=\alpha X$. Then using it in the definition of Lie differentiation we get
\begin{eqnarray}
  (\mathcal{L}_{\xi}g)(X,Y) &=& g(\nabla_X\xi,Y)+g(X,\nabla_Y\xi) \nonumber\\
   &=& g(\alpha X,Y)+g(X,\alpha Y) \nonumber\\
   &=& 2\alpha g(X,Y)
\end{eqnarray}
for all vector fields $X,Y\in\mathfrak{X}(M)$. Again, $(M,g,\mu,\xi)$ being a conformal Ricci soliton, using the value from $(4.3)$ in the equation $(1.1)$ we get
\begin{equation}
  Ric(X,Y)=(\mu-2\alpha)g(X,Y),
\end{equation}
for all vector fields $X,Y\in\mathfrak{X}(M)$, and $\mu=[2\lambda -(p+\frac{2}{n})]$. Thus equation $(4.4)$ proves that $(M,g)$ Einstein with factor $(\mu-2\alpha)$ and this completes the first part of the theorem.\par
\medskip
Again we know that for conformal Ricci flow, the scalar curvature $r(g)=-1$. So taking an orthonormal basis $\{ e_i : 1\leq i\leq n\}$ of the manifold $M$ and summing over $1\leq i\leq n$ in both sides of the equation $(4.4)$ gives us
\begin{equation}
  -1=r(g)=n(\mu-2\alpha).\nonumber
\end{equation}
Finally using the value $\mu=[2\lambda -(p+\frac{2}{n})]$ in the above equation and after simplification we get
\begin{equation}
 \lambda=\alpha+\frac{p}{2}+\frac{1}{2n}.
\end{equation}
We know that the soliton is expanding steady or shrinking if $\lambda <0$, $\lambda =0$ or $\lambda >0$, thus applying it in equation $(4.5)$ completes the proof.
\end{proof}
Next, we have a result on concurrent vector field, which immediately follows from the above theorem.
\begin{cor}
  $(M,g,\mu,\xi)$ be a conformal Ricci soliton with the soliton vector field $\xi$ is concurrent, then
  \begin{enumerate}
       \item the manifold $(M,g)$ is an Einstein manifold with factor $(\mu-2)$ and
       \item the soliton is expanding, steady or shrinking according as $(p+2+\frac{1}{n})<0$, $(p+2+\frac{1}{n})=0$ or $(p+2+\frac{1}{n})>0$ respectively.
     \end{enumerate}
   \end{cor}
\begin{proof}
  Proceeding similarly as theorem $4.1$ and then putting $\alpha=1$ in the equations $(4.4)$ and $(4.5)$ completes the proof.
\end{proof}
We conclude this section with the following theorem on concurrent vector field
\begin{thm}
   Assume that $(M,g)=(B\times_fF,g_B\oplus f^2g_F)$ be an warped product of two Riemannian manifolds $(B,g_B)$ and $(F,g_F)$ with warping function $f$, dim$B=m$ and dim$F=n$. Let $(M,g,\mu,\xi)$ be a conformal Ricci soliton with concurrent vector field $\xi$. If $f$ is constant and both the lifts $\xi_B$ and $\xi_F$ are concurrent on the base $B$ and the fiber $F$ then
  \begin{enumerate}
       \item the soliton $(M,g,\mu,\xi)$ is expanding, steady or shrinking according as $(\frac{p}{2}+\frac{1}{n}+1)<0$, $(\frac{p}{2}+\frac{1}{n}+1)=0$ or $(\frac{p}{2}+\frac{1}{n}+1)>0$ respectively,
       \item all the three manifolds $M, B$ and $F$ are Ricci flat manifolds and
       \item all the three manifolds $M, B$ and $F$ admit conformal gradient Ricci solitons.
     \end{enumerate}
   \end{thm}
\begin{proof}
  Since $(M,g,\mu,\xi)$ is a conformal Ricci soliton on $M$ with concurrent vector field $\xi$, from first part of the corollary $4.1$ we can write
  \begin{equation}
  Ric(X,Y)=(\mu-2)g(X,Y),
\end{equation}
for all vector fields $X,Y\in\mathfrak{X}(M)$.\\
Now if we set $X=X_F$ and $Y=Y_F$, then from lemma $1.1$ and equation $(2.3)$ we get
\begin{equation}
  Ric(X_F,Y_F)=Ric^F(X_F,Y_F)-\tilde{f}g_F(X_F,Y_F),
\end{equation}
where $\tilde{f}=f\Delta f+(n-1)\|\nabla f\|^2_B$. Now using equation $(4.6)$ and $(1.2)$, in the above equation $(4.7)$ yields
\begin{equation}
  Ric^F(X_F,Y_F)=\tilde{f}g_F(X_F,Y_F)+(\mu-2)f^2g_F(X_F,Y_F),\nonumber
\end{equation}
where $\tilde{f}=f\Delta f+(n-1)\|\nabla f\|^2_B$. Since it is given that $f$ is constant, say $f=c$ for some constant $c$, then it implies that $\tilde{f}=0$ and thus the above equation becomes
\begin{equation}
  Ric^F(X_F,Y_F)=c^2(\mu-2)g_F(X_F,Y_F),
\end{equation}
for all vector fields $X_F,Y_F\in\mathfrak{X}(F)$. Thus from the above equation $(4.8)$ we can say that $F$ is Einstein. Now as the equation $(4.8)$ is true for any vector field in $\mathfrak{X}(F)$, by putting $X_F=Y_F=\xi_F$ in above we get
\begin{eqnarray}
  Ric^F(\xi_F,\xi_F) &=& c^2(\mu-2)g_F(\xi_F,\xi_F) \nonumber\\
  &=& c^2(\mu-2)\|\xi_F\|^2_F.
\end{eqnarray}
Let $\{ \xi_F, e_1, e_2, e_3,...., e_{n-1}\}$ be an orthonormal basis of $\mathfrak{X}(F)$. Then the curvature tensor of the manifold $F$ is given by
\begin{equation}
  R^F(\xi_F,e_i,\xi_F,e_i)=g_F(R^F(\xi_F,e_i)\xi_F,e_i). \nonumber
\end{equation}
Using the well-known formula for curvature, the above equation can be rewritten as
\begin{equation}
 R^F(\xi_F,e_i,\xi_F,e_i)= g_F(\nabla ^F_{\xi_F}\nabla ^F_{e_i}\xi_F-\nabla ^F_{e_i}\nabla ^F_{\xi_F}\xi_F
 -\nabla ^F_{[\xi_F,e_i]}\xi_F, e_i).
\end{equation}
Also since $\xi_F$ is concurrent vector field, from equation $(4.2)$ we have $\nabla_X\xi_F=X$, for all $X\in\mathfrak{X}(F)$ and using this in equation $(4.10)$ we get
\begin{equation}
 R^F(\xi_F,e_i,\xi_F,e_i)= g_F(\nabla ^F_{\xi_F}{e_i}-\nabla ^F_{e_i}\xi_F-[\xi_F,e_i], e_i)=0.\nonumber
\end{equation}
This implies $Ric^F(\xi_F,\xi_F)=0$ and then from equation $(4.9)$ we get $\mu=2$, i.e; $\mu=[2\lambda -(p+\frac{2}{n})]=2$. After simplification this gives $\lambda=(\frac{p}{2}+\frac{1}{n}+1)$ and the soliton is shrinking, steady or expanding according as $\lambda >0$, $\lambda =0$ or $\lambda <0$. This proves the first part of the theorem.\par
\medskip
 Now, using this value $\mu=2$ in equations $(4.8)$ and $(4.6)$ we have $Ric=Ric^F=0$. This proves that both the manifolds $M$ and $F$ are Ricci flat.\\
Again if we set $X=X_B$ and $Y=Y_B$, then from lemma $1.1$ we can write
\begin{equation}
Ric(X_B,Y_B)=Ric^B(X_B,Y_B)-\frac{n}{f}H^f(X_B,Y_B),\nonumber
\end{equation}
for all $X_B,Y_B\in\mathfrak{X}(B)$. Now since we just proved $Ric=0$, the above equation becomes
\begin{equation}
Ric^B(X_B,Y_B)=\frac{n}{f}H^f(X_B,Y_B).
\end{equation}
Since we assumed that $f$ is constant, it implies $H^f=0$ and thus the above equation $(4.11)$ finally gives us $Ric^B(X_B,Y_B)=0$, for all $X_B,Y_B\in\mathfrak{X}(B)$. Therefore we get $Ric^B=0$ and this proves that the manifold $B$ is Ricci flat. This completes the proof of the second part of the theorem.\par
\medskip
To prove the last part of the theorem, let us assume that $\phi=\frac{1}{2}g(\xi,\xi)$. Then
\begin{equation}
g(X, grad\phi)=X(\phi)=g(\nabla_X\xi,\xi),
\end{equation}
for all $X\in\mathfrak{X}(M)$. Again $\xi$ being concurrent, from equation $(4.2)$ we have $\nabla_X\xi=X$ and using this value in equation $(4.12)$ we get
\begin{equation}
g(X, grad\phi)=g(X,\xi)\nonumber
\end{equation}
for all $X\in\mathfrak{X}(M)$. Since the above equation is true for any vector field $X\in\mathfrak{X}(M)$, we can conclude that $\xi=grad\phi$. Hence $(M,g)$ admits a conformal gradient Ricci soliton.\\
 Again taking $\phi_B=\frac{1}{2}g(\xi_B,\xi_B)$ and $\phi_F=\frac{1}{2}g(\xi_F,\xi_F)$ and proceeding similarly we can show that $\xi_B=grad\phi_B$ and $\xi_F=grad\phi_F$. Also from theorem $2.1$ we know that since $(M,g)$ is conformal Ricci soliton, $B$ and $F$ both are conformal Ricci soliton. Hence can conclude that both the manifolds $B$ and $F$ admit conformal gradient Ricci soliton.
 \end{proof}
 \medskip
\medskip
\section{\textbf{Application of conformal Ricci soliton on generalized Robertson-Walker spacetimes}}
This section deals with the study of conformal Ricci solitons on a very well-known warped spacetime called a generalized Robertson-Walker spacetime which is an extension of the classical Robertson-Walker spacetimes. It is to be noted that, generalized Robertson-Walker spacetimes also obey the Weyl hypothesis, i.e; the world lines should be everywhere orthogonal to a family of spacelike slices. M. S\'{a}nchez \cite{[10]} characterized generalized Robertson-Walker spacetimes in terms of timelike and spatially conformal conformal vector fields. Also, a characterization of generalized Robertson-Walker spacetimes in terms of timelike concircular vector field has been studied by B.Y. Chen \cite{[17]}.
\begin{defn}
  A generalized Robertson-Walker spacetime is a warped product manifold $M=I\times_f F$ endowed with the Lorentzian metric
  \begin{equation}
    g=-dt^2\oplus f^2g_F,
  \end{equation}
  where the base is an open interval $I$ of $\mathbb{R}$ with its usual metric reversed $(I,-dt^2)$, the fiber is an $n$-dimensional Riemannian manifold $(F,g_F)$ and the warping function is any positive function $f>0$ on $I$.
\end{defn}
\par
\medskip
Based on the above definiton we will consider a generalized Robertson-Walker spacetime and study the effect of conformal Ricci soliton on it. Our main result of this section is the following
\begin{thm}
  Let $M=I\times_f F$ be a generalized Robertson-Walker spacetime endowed with the metric $g=-dt^2\oplus f^2g_F$ and let $\phi=\int_{c}^{t}f(z)dz$, for some constant $c\in I$. If $(M,g,\mu,\phi)$ admits a conformal gradient Ricci soliton, then
  \begin{enumerate}
    \item the generalized Robertson-Walker spacetime $(M,g)$ becomes Ricci flat if the soliton constant $\lambda$ satisfies the relation $\lambda=\dot{f}+\frac{p}{2}+\frac{1}{n}$ and
    \item the generalized Robertson-Walker spacetime $(M,g)$ is an Einstein manifold if the warping function $f$ is of the form $f(t)=at+b$, where $a,b$ are constants.
  \end{enumerate}
\end{thm}
\begin{proof}
  As per our assumption, $(M,g,\mu,\phi)$ being a conformal gradient Ricci soliton, setting $\xi=grad\phi$, from equation $(1.1)$ we can write
  \begin{equation}
(\mathcal{L}_{\xi}g)(X,Y)+2Ric(X,Y)=\mu g(X,Y)=[2\lambda -(p+\frac{2}{n})]g(X,Y),
\end{equation}
for all $X,Y\in\mathfrak{X}(M)$.\\
  Again since  $\phi=\int_{c}^{t}f(z)dz$, then $\xi=grad\phi$ implies that $\xi=f(t)\frac{\partial}{\partial t}$ and it can be seen that the vector field $\xi$ is orthogonal to the manifold $F$.\\
 Let us assume that $\{\frac{\partial}{\partial t}, \frac{\partial}{\partial x_1}, \frac{\partial}{\partial x_2},...., \frac{\partial}{\partial x_n}\}$ be an orthonormal basis of $\mathfrak{X}(M)$. Then the Hessian of the function $\phi$ is given by
 \begin{equation}
   H^{\phi}(X,Y)=g(\nabla_Xgrad\phi,Y).
 \end{equation}
 Now, we consider the following three cases.\par
 \medskip
 \textbf{Case 1:} First let us consider $X=Y=\frac{\partial}{\partial t}$.\\
 Then from equation $(5.3)$ we get
 \begin{eqnarray}
   H^{\phi}(\frac{\partial}{\partial t},\frac{\partial}{\partial t}) &=& g(\nabla_{\frac{\partial}{\partial t}}grad\phi,\frac{\partial}{\partial t}) \nonumber\\
    &=& \dot{f}g(\frac{\partial}{\partial t},\frac{\partial}{\partial t}).
 \end{eqnarray}
 \par
 \medskip
 \textbf{Case 2:} Next we consider $X=\frac{\partial}{\partial t}$ and $Y=\frac{\partial}{\partial x_i}$ for $i=1,2,....,n$.\\
 Then in this case equation $(5.3)$ implies
 \begin{eqnarray}
   H^{\phi}(\frac{\partial}{\partial t},\frac{\partial}{\partial x_i}) &=& g(\nabla_{\frac{\partial}{\partial t}}grad\phi,\frac{\partial}{\partial x_i}) \nonumber\\
    &=& \dot{f}g(\frac{\partial}{\partial t},\frac{\partial}{\partial x_i}).
 \end{eqnarray}
 \par
 \medskip
 \textbf{Case 3:} Finally we consider $X=\frac{\partial}{\partial x_i}$ and $Y=\frac{\partial}{\partial x_j}$ for $1\leq i,j\leq n$.\\
 Then from equation $(5.3)$ we have
 \begin{eqnarray}
   H^{\phi}(\frac{\partial}{\partial x_i},\frac{\partial}{\partial x_j}) &=& g(\nabla_{\frac{\partial}{\partial x_i}}grad\phi,\frac{\partial}{\partial x_j}) \nonumber\\
    &=& fg(\nabla_{\frac{\partial}{\partial x_i}}\frac{\partial}{\partial t},\frac{\partial}{\partial x_j}) \nonumber\\
    &=& fg(\frac{\dot{f}}{f}\frac{\partial}{\partial x_i},\frac{\partial}{\partial x_j}) \nonumber\\
    &=& \dot{f}g(\frac{\partial}{\partial x_i},\frac{\partial}{\partial x_j}).
 \end{eqnarray}
 \par
 \medskip
 Therefore combining equations $(5.4)$, $(5.5)$ and $(5.6)$ and using it in $(5.3)$ we get
 \begin{equation}
   H^{\phi}(X,Y)=\dot{f}g(X,Y).
 \end{equation}
 Now, since $\xi=grad\phi$, using the definition of Lie differentiation we can write
 \begin{eqnarray}
    (\mathcal{L}_{\xi}g)(X,Y) &=& g(\nabla_Xgrad\phi,Y)+g(X,\nabla_Ygrad\phi)\nonumber\\
                              &=& 2H^{\phi}(X,Y).\nonumber
 \end{eqnarray}
Thus using equation $(5.7)$, the above equation becomes
 \begin{equation}
   (\mathcal{L}_{\xi}g)(X,Y)=2\dot{f}g(X,Y).
 \end{equation}
 Using the value of equation $(5.8)$ in the equation $(5.2)$ and after simplification we get
 \begin{equation}
 Ric(X,Y)=[\lambda -\dot{f}-\frac{p}{2}-\frac{1}{n}]g(X,Y).
 \end{equation}
 Thus if $\lambda=\dot{f}+\frac{p}{2}+\frac{1}{n}$, from equation $(5.9)$, it implies that $(M,g)$ is Ricci flat. This completes the first part of the theorem.\par
 \medskip
 Again if $\dot{f}$ is a constant, say $\dot{f}=a$, i.e; if $df=adt$, i.e; if $f=at+b$ for some arbitrary constant $b$, then from equation $(5.9)$ we can conclude that $(M,g)$ is Einstein. This completes the proof.
\end{proof}
\par

\end{document}